\documentclass[reqno]{amsart}
\usepackage{amssymb, amsmath, amsthm, enumerate, mathtools, graphicx, enumitem, tikz, color, soul, hyperref}
\usepackage{stmaryrd}
\usepackage[margin=1 in]{geometry}
\usepackage[mathscr]{euscript}
\allowdisplaybreaks

\usepackage{amsmath}
\usepackage{verbatim}

\DeclareFontFamily{U}{mathx}{\hyphenchar\font45}
\DeclareFontShape{U}{mathx}{m}{n}{
      <5> <6> <7> <8> <9> <10>
      <10.95> <12> <14.4> <17.28> <20.74> <24.88>
      mathx10
      }{}
\DeclareSymbolFont{mathx}{U}{mathx}{m}{n}
\DeclareFontSubstitution{U}{mathx}{m}{n}
\DeclareMathAccent{\widecheck}{0}{mathx}{"71}
\DeclareMathAccent{\wideparen}{0}{mathx}{"75}

\usepackage{colonequals}

\newcommand{\fE}{\mathcal{E}}

\newcommand{\fI}{\mathcal{I}}

\newcommand{\C}{\mathbb{C}}
\newcommand{\R}{\mathbb{R}}

\newcommand{\bbS}{\mathbb{S}}

\newcommand{\supp}{\operatorname{supp}}

\newcommand{\fP}{\mathcal{P}}

\newcommand{\fQ}{\mathcal{Q}}

\newcommand{\dist}{\operatorname{dist}}
\newcommand{\ang}{\operatorname{ang}}

\newcommand{\fru}{\operatorname{fru}}

\newcommand{\ellip}{\operatorname{ell}}

\newcommand{\llangle}{{\langle\!\langle}}
\newcommand{\rrangle}{{\rangle\!\rangle}}
\newcommand{\bN}{\boldsymbol{N}}

\def\XXint#1#2#3{{\setbox0=\hbox{$#1{#2#3}{\int}$}
     \vcenter{\hbox{$#2#3$}}\kern-.5\wd0}}

\newtheorem{theorem}{Theorem}[section]
\newtheorem{proposition}[theorem]{Proposition}
\newtheorem{lemma}[theorem]{Lemma}
\newtheorem{conjecture}[theorem]{Conjecture}
\newtheorem{corollary}[theorem]{Corollary}
\theoremstyle{definition}

\numberwithin{equation}{section}

\begin{document}
\title{Global restriction estimates for elliptic hyperboloids}
\address{Department of Mathematics, University of Wisconsin, Madison}
\email{bbruce@math.wisc.edu}
\author{Benjamin Baker Bruce}
\date{\today}
\maketitle

\begin{abstract}
We prove global Fourier restriction estimates for elliptic, or two-sheeted, hyperboloids of arbitrary dimension $d \geq 2$, extending recent joint work with Oliveira e Silva and Stovall.  Our results are unconditional in the (adjoint) bilinear range, $q > \frac{2(d+3)}{d+1}$, and extend conditionally upon further progress toward the local restriction conjecture for elliptic surfaces. 
\end{abstract}

\section{Introduction}\label{intro}
In this article, we establish global Fourier restriction estimates for elliptic, or two-sheeted, hyperboloids of arbitrary dimension $d \geq 2$.  These surfaces take the general form
\begin{align}\label{generic hyperboloid}
\{(\tau,\xi) \in \R \times \R^d \colon (\tau - \tau_0, \xi - \xi_0) \cdot A (\tau-\tau_0,\xi-\xi_0) = 1\},
\end{align}
where $(\tau_0,\xi_0) \in \R \times \R^d$ and $A$ is an invertible $(d+1) \times (d+1)$ matrix with exactly one positive eigenvalue.  Due to the affine invariance of restriction estimates and time-reversal symmetry, we may (and will) restrict our attention to the surface 
\begin{align*}
\Sigma \colonequals \{(\tau,\xi) \in \R \times \R^d \colon \tau = \langle \xi \rangle\}, \quad\quad\quad \langle \xi \rangle \colonequals \sqrt{|\xi|^2 + 1},
\end{align*}
which is the ``upper sheet" of \eqref{generic hyperboloid} with $(\tau_0,\xi_0) = (0,0)$ and $A = \operatorname{diag}(1,-1,\ldots,-1)$.

We begin by describing the context for this project.  While certain aspects of the restriction theory for hyperboloids have already been studied, see e.g.~\cite{Strichartz}, \cite{Quilodran}, \cite{Carneiro--Oliveira e Silva--Sousa}, \cite{Carneiro--Oliveira e Silva--Sousa--Stovall}, the question of the optimal range of global estimates was only very recently taken up by Oliveira e Silva, Stovall, and the author in \cite{Bruce--Oliveira e Silva--Stovall}.  There, the hyperbolic, or one-sheeted, hyperboloid in three ambient dimensions was studied.  The present article generalizes certain techniques from \cite{Bruce--Oliveira e Silva--Stovall} to obtain global restriction estimates for higher-dimensional hyperboloids.  As noted above, our results will be stated and proved for elliptic hyperboloids, whose local restriction theory has been well studied (see e.g.~\cite{Tao}, \cite{Hickman--Rogers}); however, similar methods could potentially yield purely conditional results for hyperbolic hyperboloids.

Hyperboloids are geometrically interesting from the viewpoint of restriction theory. Historically, a significant proportion of the work on restriction has focused on compact surfaces, such as the unit sphere or the truncated paraboloid.  Indeed, a general form of the restriction conjecture asserts, in part, that every smooth compact surface with at least one nonvanishing principal curvature admits some nontrivial restriction estimate.  As is well known, however, homogeneity can sometimes be substituted for compactness.  Paraboloids and cones are prototypical examples of noncompact surfaces that obey homogeneity relations and admit nontrivial restriction estimates.  While hyperboloids are not homogeneous in this sense, they come close by ``interpolating" paraboloids and cones:  the surface $\Sigma$, for example, resembles the paraboloid $\tau = \frac{1}{2}|\xi|^2 +1$ as $|\xi| \rightarrow 0$ and the cone $\tau = |\xi|$ as $|\xi| \rightarrow \infty$.  As we will show, hyperboloids appear to admit a range of restriction estimates that interpolates (in a precise sense) the restriction conjectures for paraboloids and cones, and our proof will rest on adaptations of the bilinear restriction theories associated to those surfaces.   Also crucial to our arguments will be the invariance of hyperboloids under appropriately defined Lorentz transformations.  Certain of these transformations, sometimes termed ``orthochronous," additionally preserve each sheet of the two-sheeted hyperboloid.  The orthochronous Lorentz group that acts (transitively) on $\Sigma$ consists of linear maps on $\R \times \R^d$ that preserve the quadratic form $(\tau,\xi) \mapsto \tau^2 - |\xi|^2$ as well as $(\tau,\xi) \mapsto \operatorname{sign} \tau$.  Throughout this article, ``Lorentz" will always mean ``orthochronous Lorentz," so that Lorentz transformations are symmetries of $\Sigma$.

We now turn to the basic definitions and statements of our results.    To begin, we equip $\Sigma$ with its unique Lorentz-invariant measure $\mu$, given by
\begin{align*}
\int_{\Sigma} f d\mu \colonequals \int_{\R^d}f(\langle \xi \rangle, \xi)\frac{d\xi}{\langle \xi \rangle}.
\end{align*}
The Gaussian curvature of $\Sigma$ at the point $(\langle \xi \rangle,\xi)$ is equal to $\langle \xi \rangle^{-d-2}$, and thus $\mu$ coincides with the so-called affine surface measure on $\Sigma$.  The role of $\mu$ is twofold:  to respect the symmetries of our surface and to compensate for its degenerating curvature.  As is standard, we will formulate our results in terms of the adjoint restriction, or extension, operator.   Having equipped $\Sigma$ with $\mu$, this operator takes the form
\begin{align*}
\fE f(t,x) \colonequals \widecheck{f d\mu}(t,x) = \int_{\Sigma} e^{i(t,x)\cdot(\tau,\xi)}f(\tau,\xi)d\mu(\tau,\xi)
\end{align*}
for $f$ continuous and compactly supported.  Henceforth, the dual terms ``restriction" and ``extension" will be used interchangeably.  We denote by
\begin{align*}
\Sigma_0 \colonequals \{(\tau,\xi) \in \Sigma \colon |\xi| \leq 2\} \quad\quad\text{and}\quad\quad \fE_0 f \colonequals \fE(f\chi_{\Sigma_0})
\end{align*}
the low-frequency ``parabolic" region of $\Sigma$ and the corresponding local extension operator.  Given $p,q \in [1,\infty]$, we denote by $\fE(p \rightarrow q)$ the statement that $\fE$ extends to a bounded linear operator from $L^p(\Sigma,\mu)$ to $L^q(\R \times \R^d)$, and we write $\fE_0(p \rightarrow q)$ if the analogous statement holds for $\fE_0$.  We can now state a conjecture on the complete range of exponent pairs $(p,q)$ for which $\fE(p \rightarrow q)$ is valid.

\begin{conjecture}\label{conj}
 If $(\frac{d}{d+2}q)' \leq p \leq \min\{(\frac{d-1}{d+1}q)', q\}$ and $(p,q) \neq (\frac{2d}{d-1},\frac{2d}{d-1}), (\frac{2(d+1)}{d},\frac{2(d+1)}{d})$, then $\fE(p \rightarrow q)$ holds.
\end{conjecture}

\noindent The main results of this article are the following:

\begin{theorem}\label{unconditional result}
If $q > \max\{\frac{2(d+3)}{d+1},p\}$ and $(\frac{d}{d+2}q)' \leq p \leq (\frac{d-1}{d+1}q)'$, then $\fE(p \rightarrow q)$ holds.  If $d= 2$, then additionally $\fE(q \rightarrow q)$ holds for $\frac{10}{3} < q < 4$.
\end{theorem}

\begin{theorem}\label{conditional result}
If $\fE_0(p_0 \rightarrow q_0)$ holds with $p_0' = \frac{d}{d+2}q_0$ for some $q_0 < \frac{2(d+3)}{d+1}$, then $\fE(p \rightarrow q)$ holds for all $(p,q)$ obeying $\max\{q_0,p\} < q < \frac{2(d+3)}{d+1}$ and $(\frac{d}{d+2}q)' \leq p \leq (\frac{d-1}{d+1}q)'$ and
\begin{align}\label{condition on exponents}
\frac{1}{p} > \frac{2}{(d-1)(d+3)}\Bigg(\frac{\frac{1}{q}-\frac{d+1}{2(d+3)}}{\frac{1}{q_0}-\frac{d+1}{2(d+3)}} + \frac{d^2+2d-7}{4}\Bigg).
\end{align}
\end{theorem}

The conditional result, Theorem \ref{conditional result}, may warrent some explanation.  It implies, in particular, that any optimal local extension estimate beyond the bilinear range, i.e.~ $\fE_0(p_0 \rightarrow q_0)$ with $p_0' =  \frac{d}{d+2}q_0$ for some $q_0 < \frac{2(d+3)}{d+1}$, would lead to an improvement of our unconditional result, Theorem \ref{unconditional result}.  The requirement that $p_0' = \frac{d}{d+2}q_0$ is a limitation, but it seems necessary (given our techniques) for obtaining any additional global estimates on the parabolic scaling line $p' = \frac{d}{d+2}q$.  If we assume only that $\fE_0(p_0 \rightarrow q_0)$ holds for some $p_0 > (\frac{d}{d+2}q_0)'$, then a slight modification of the proof of Theorem \ref{conditional result} would likely yield a range of global estimates that excludes the line $p' = \frac{d}{d+2}q$ beyond the bilinear range but nevertheless improves Theorem \ref{unconditional result}.

The rest of the article is organized as follows:  In Section \ref{negative result}, we prove a negative result, that Conjecture \ref{conj} cannot be improved.  In Section \ref{bilinear theory}, we present the bilinear restriction theory for the ``conic" portion of our surface $\Sigma$, and we deduce uniform linear extension estimates on dyadic frusta from bilinear estimates between certain thin sectors.  (These results resemble the bilinear restriction theory for cones.)  In Section \ref{summing uniform bounds}, we use a Strichartz inequality for the Klein--Gordon equation to sum the uniform estimates on frusta, and consequently we obtain Theorem \ref{unconditional result}.  In Section \ref{decoupling}, we use the conic decoupling theorem of \cite{Bourgain--Demeter} to convert conditional local extension estimates into uniform estimates on frusta, and then, appealing to Section \ref{summing uniform bounds} again, we obtain Theorem \ref{conditional result}.  Finally, in Section \ref{improvements}, we discuss possible improvements to our results by means of the state-of-the art local extension estimates for elliptic surfaces.

\medskip
{\bf Notation.} We use the standard notations $A \lesssim B$ and $A = O(B)$ to mean that $A \leq CB$ for some constant $C > 0$.  If this constant depends on some parameter $\varepsilon$, then we might write $A \lesssim_\varepsilon B$.  Typically, constants may only depend on the dimension $d$ and relevant Lebesgue space exponents.

\medskip
{\bf Acknowledgments.} The author thanks Betsy Stovall for suggesting this project and acknowledges support from NSF grant DMS-1653264.


\section{Optimality of Conjecture \ref{conj}}\label{negative result}

In this section, we demonstrate that Conjecture \ref{conj} cannot be improved.  Although the counterexamples we consider are well known, we include all necessary details.

\begin{proposition}
If $\fE(p \rightarrow q)$ holds, then $(p,q)$ satisfies the hypotheses of Conjecture \ref{conj}.
\end{proposition}
\begin{proof}
Assume that $\fE(p \rightarrow q)$ holds.  To show that $p \geq (\frac{d}{d+2}q)'$, we use the standard Knapp example.  Fixing $ \delta \in (0,1]$, let 
\begin{align*}
C \colonequals \{(\tau, \xi) \in \Sigma \colon  |\xi| \leq \delta\}
\end{align*}
and
\begin{align*}
T \colonequals \{(t,x) \in \R \times \R^d \colon |t| \leq c\delta^{-2},~|x| \leq c\delta^{-1}\},
\end{align*}
where $c$ is a small constant.  If $(t,x) \in T$ and $c$ is sufficiently small, then
\begin{align*}
|\fE\chi_{C}(t,x)| = \bigg\vert\int_{C} e^{i(t,x)\cdot(\tau-1,\xi)}d\mu(\tau,\xi)\bigg\vert \geq \bigg\vert \int_{C} \cos(t(\tau-1) + x\cdot \xi)d\mu(\tau,\xi)\bigg\vert \sim \mu(C) \sim \delta^d.
\end{align*}
Therefore,
\begin{align*}
\delta^{d-\frac{d+2}{q}} \sim \delta^d|T|^\frac{1}{q} \lesssim \|\fE\chi_{C}\|_q \lesssim \mu(C)^\frac{1}{p} \sim \delta^\frac{d}{p}
\end{align*}
by the validity of $\fE(p \rightarrow q)$.  Letting $\delta \rightarrow 0$, we conclude that $p \geq (\frac{d}{d+2}q)'$.

The necessity of $p \leq (\frac{d-1}{d+1}q)'$ follows by a similar argument, using that $\mu$ degenerates away from the origin.  Indeed, fixing $\lambda \geq 1$, let
\begin{align*}
D \colonequals \{(\tau,\xi) \in \Sigma \colon |\xi| \leq \lambda\}.
\end{align*}
If $|(t,x)| \leq c\lambda^{-1}$ and $c$ is sufficiently small, then
\begin{align*}
|\fE\chi_{D}(t,x)| \sim \mu(D) \sim \lambda^{d-1}.
\end{align*}
Thus, 
\begin{align*}
\lambda^{d-1-\frac{d+1}{q}} \lesssim \|\fE\chi_{D}\|_q \lesssim \mu(D)^\frac{1}{p} \sim \lambda^\frac{d-1}{p}
\end{align*}
by $\fE(p \rightarrow q)$, and sending $\lambda \rightarrow \infty$ gives the required inequality.

To show that $p \leq q$, we consider a randomized sum of bump functions.  By interpolation, we may assume that $q \geq 2$.  Let $\phi$ be a nonzero bump function on $\Sigma$ and $N$ a positive integer.  For each $j \in \{1,\ldots,N\}$, let $L_j$ be a Lorentz boost and set $\phi_j \colonequals \phi \circ L_j$.  Choosing the boosts $L_j$ appropriately, the functions $\phi_j$ have pairwise disjoint supports.  Let $f \colonequals \sum_{j=1}^N \varepsilon_j \phi_j$, where $\varepsilon_1,\ldots, \varepsilon_N$ are independent random variables with the Rademacher distribution.  On one hand, Khintchine's inequality and the Lorentz invariance of $\mu$ imply that
\begin{align*}
\mathbb{E}\|\fE f\|_q^q &= \iint_{\R \times \R^d}\mathbb{E}\bigg\vert\sum_{j=1}^N \varepsilon_j \fE\phi_j(t,x)\bigg\vert^q dtdx\\
&\gtrsim \iint_{\R \times \R^d}\bigg(\sum_{j=1}^N|\fE \phi_j(t,x)|^2\bigg)^\frac{q}{2}dtdx\\
&\geq \sum_{j=1}^N\iint_{\R \times \R^d}|\fE\phi_j(t,x)|^qdtdx\\
&=N\|\fE\phi\|_q^q.
\end{align*}
On the other hand, 
\begin{align*}
\mathbb{E}\|\fE f\|_q^q \lesssim \mathbb{E}\|f\|_p^q = N^\frac{q}{p}\|\phi\|_p^q
\end{align*}
by $\fE(p \rightarrow q)$, the fact that the $\phi_j$ have disjoint supports, and Lorentz invariance.  Letting $N \rightarrow \infty$ shows that $p \leq q$.

Finally, setting $p_1 \colonequals \frac{2d}{d-1}$ and $p_2 \colonequals \frac{2(d+1)}{d}$, we need to show that the estimates $\fE(p_1 \rightarrow p_1)$ and $\fE(p_2 \rightarrow p_2)$ are false.  We could proceed by rescaling known counterexamples for the cone and paraboloid.  Indeed, $(p_1,p_1)$ lies on the conic scaling line $p' = \frac{d-1}{d+1}q$ and the extension operator associated to the cone is known not to be bounded on $L^{p_1}$; likewise, $(p_2,p_2)$ lies on the parabolic scaling line $p' = \frac{d}{d+2}q$ and the extension operator for the paraboloid is not bounded on $L^{p_2}$.  
We will instead present direct counterexamples to $\fE(p_1 \rightarrow p_1)$ and $\fE(p_2 \rightarrow p_2)$, using longer but self-contained arguments.

We start with the disproof of $\fE(p_1 \rightarrow p_1)$.  Fixing $\lambda \geq 1$, let $f \colon \Sigma \rightarrow \C$ be defined by $f(\tau,\xi) \colonequals \psi(|\xi|)\langle \xi \rangle$, where $\psi$ is a bump function satisfying $\chi_{[2\lambda, 3\lambda]} \leq \psi \leq \chi_{[\lambda, 4\lambda]}$ and $|\psi'| \lesssim \lambda^{-1}$.  Using polar coordinates, we see that
\begin{align*}
\fE f(t,x) = \int_{\lambda}^{4\lambda}\int_{\bbS^{d-1}}e^{i(t,x)\cdot(\langle r \rangle, r\theta)}\psi(r){r^{d-1}}d\sigma(\theta)dr = \int_\lambda^{4\lambda} \check{\sigma}(rx) e^{it\langle r \rangle}\psi(r){r^{d-1}} dr,
\end{align*}
where $\langle r \rangle \colonequals \sqrt{r^2+1}$ and $\sigma$ is the standard measure on the sphere $\bbS^{d-1}$.  By a well-known stationary phase argument (see e.g.~\cite{Wolff}), $\check{\sigma}$ obeys the asymptotic formula
\begin{align*}
\check{\sigma}(y) = a|y|^{-\frac{d-1}{2}}\cos(|y|+b) + O(|y|^{-\frac{d+1}{2}}) \quad\quad\text{as }|y| \rightarrow \infty
\end{align*}
for some $a,b \in \R$ with $a > 0$.  Thus,
\begin{align*}
|\fE f(t,x)| = a|x|^{-\frac{d-1}{2}}\int_{\lambda}^{4\lambda}\cos(r|x| + b)e^{i(t\langle r \rangle + b)}\psi(r){r^{\frac{d-1}{2}}}dr + O(\lambda^{\frac{d-1}{2}}|x|^{-\frac{d+1}{2}}),
\end{align*}
provided $|x| \geq 1$ and $\lambda$ is sufficiently large.   The absolute value of the integral is at least that of its real part.  Using the identity $2\cos(\theta)\cos(\nu) = \cos(\theta -\nu)+\cos(\theta+\nu)$, we get the bound
\begin{align}\label{asymptotic lower bound}
|\fE f(t,x)| \geq \frac{a}{2}|x|^{-\frac{d-1}{2}}(|I_1| - |I_2|) + O(\lambda^{\frac{d-1}{2}}|x|^{-\frac{d+1}{2}}),
\end{align}
where
\begin{align*}
I_1 &\colonequals \int_\lambda^{4\lambda}\cos(r|x| - t\langle r \rangle)\psi(r){r^{\frac{d-1}{2}}}dr,\\
I_2 &\colonequals \int_\lambda^{4\lambda}\cos(r|x| + t\langle r \rangle + 2b)\psi(r){r^{\frac{d-1}{2}}}dr.
\end{align*}
Suppose that $1 \leq |x| \leq c\lambda$ and $|t - |x|| \leq c\lambda^{-1}$, where $c > 0$ is a constant.  If $c$ is sufficiently small, then $|r|x|-t\langle r \rangle| \leq 1$ for all $r \in [\lambda, 4\lambda]$, leading to the bound
\begin{align*}
|I_1| \gtrsim \lambda^\frac{d+1}{2}.
\end{align*}
To estimate $I_2$, we first write $I_2 = I_{21} + I_{22}$, where
\begin{align*}
I_{21} &\colonequals \int_{\lambda}^{4\lambda}\cos(r(|x|+t)+2b)\psi(r){r^\frac{d-1}{2}}dr,\\
I_{22} &\colonequals \int_{\lambda}^{4\lambda}(\cos(r|x|+t\langle r \rangle + 2b) - \cos(r(|x|+t)+2b))\psi(r){r^\frac{d-1}{2}}dr.
\end{align*}
The assumptions on $(t,x)$ imply that $||x|+t| \geq 1$, giving the bound
\begin{align*}
|I_{21}| = \bigg\vert\int_{\lambda}^{4\lambda}\frac{\sin(r(|x|+t)+2b)}{|x|+t}\frac{d}{dr}(\psi(r){r^\frac{d-1}{2}})dr\bigg\vert \lesssim \lambda^\frac{d-1}{2}.
\end{align*}
Estimating the integrand of $I_{22}$ using the mean value theorem, we find that
\begin{align*}
|I_{22}| \lesssim c\lambda^\frac{d+1}{2}.
\end{align*}
Thus, $|I_1| - |I_2| \geq \frac{1}{2}|I_1| \gtrsim \lambda^\frac{d+1}{2}$ if $\lambda$ is sufficiently large and $c$ sufficiently small.  Plugging this bound into \eqref{asymptotic lower bound}, we conclude that
\begin{align*}
|\fE f(t,x)| \gtrsim \lambda^\frac{d+1}{2}|x|^{-\frac{d-1}{2}}
\end{align*}
for all $(t,x)$ satisfying $1 \leq |x| \leq c\lambda$ and $|t-|x|| \leq c\lambda^{-1}$.  If $\fE(p_1 \rightarrow p_1)$ were true, then it would follow that
\begin{align*}
\lambda^\frac{d^2+1}{d-1}\log\lambda \lesssim \|\fE f\|_{p_1}^{p_1} \lesssim \|f\|_{p_1}^{p_1} \sim \lambda^\frac{d^2+1}{d-1},
\end{align*}
and sending $\lambda \rightarrow \infty$ would give a contradiction.

The disproof of $\fE(p_2 \rightarrow p_2)$ is similar but simpler.  We will follow an argument from \cite[Chapter VIII]{Stein}.  Define $f \colon \Sigma \rightarrow \C$ by $f(\tau,\xi) \colonequals \psi(\xi)\langle \xi \rangle$, where $\psi$ is a bump function satisfying $\psi(\xi) = 1$ for $|\xi| \leq c$ and $\psi(\xi) = 0$ for $|\xi| \geq 2c$, with $c$ a positive constant.  Fix $(t,x) \in \R \times \R^d \setminus \{(0,0)\}$, and let $\lambda \colonequals |(t,x)|$ and $(t_0,x_0) \colonequals \lambda^{-1}(t,x)$.  Then
\begin{align*}
\fE f(t,x) = \int_{\R^d} e^{i\lambda \Phi(\xi;t_0,x_0)}\psi(\xi)d\xi,
\end{align*}
where $\Phi(\xi; s,y) \colonequals s\langle \xi \rangle + y\cdot \xi$.  Since $\nabla_\xi \Phi(0;1,0) = 0$ and $\det\nabla_\xi^2\Phi(0;1,0) = 1$, the implicit function theorem implies the existence of a neighborhood $U$ of $(1,0)$ such that if $(s,y) \in U$, then there exists a unique $\xi(s,y)$ such that $\nabla_\xi\Phi(\xi(s,y); s,y) = 0$.  Making $U$ smaller if necessary, we may assume that $|\xi(s,y)| \leq c$ and $|\det \nabla_\xi^2\Phi(\xi(s,y); s,y)| \gtrsim 1$ for all $(s,y) \in U$.  Suppose that $(t_0,x_0) \in U$, so that $\xi(t_0,x_0)$ is a nondegenerate critical point of the function $\Phi(\cdot; t_0,x_0)$.  If $c$ is sufficiently small (not depending on $(t,x)$), then
\begin{align*}
\bigg\vert\int_{\R^d} e^{i\lambda \Phi(\xi;t_0,x_0)}\psi(\xi)d\xi\bigg\vert = a|\det\nabla_\xi^2\Phi(\xi(t_0,x_0);t_0,x_0)|^{-\frac{1}{2}}\lambda^{-\frac{d}{2}} + O(\lambda^{-\frac{d+1}{2}})
\end{align*}
for some constant $a > 0$, by a standard stationary phase result (see e.g.~\cite[Chapter VIII, Proposition 6]{Stein}).  It follows that 
\begin{align*}
|\fE f(t,x)| \gtrsim |(t,x)|^{-\frac{d}{2}}
\end{align*}
whenever $t$ is sufficiently large and $t^{-1}|x|$ sufficiently small.  Consequently, we find that $\fE f \notin L^{p_2}(\R \times \R^d)$, and thus $\fE(p_2 \rightarrow p_2)$ cannot hold.
\end{proof}

\section{Bilinear and linear restriction on frusta}\label{bilinear theory}
Now we begin working toward our main results.  In this section, we divide the ``conic" portion of our surface, $\Sigma \setminus \Sigma_0$, into frusta $\Sigma_N$ of width $2^N$ for which we prove uniform extension estimates.  Our proof will combine the bilinear restriction theory for the frusta $\Sigma_N$ (resembling that for conic frusta), the bilinear theory for $\Sigma_0$ (resembling that for the paraboloid), and the bilinear-to-linear argument found in \cite{Tao--Vargas--Vega}.

We turn to the details.  For each integer $N \geq 1$, let
\begin{align*}
\Sigma_N \colonequals \{(\tau,\xi) \in \Sigma \colon 2^N \leq |\xi| \leq 2^{N+1}\}.
\end{align*}
Given $1 \leq k \leq N$, we cover the frustum $\Sigma_N$ by sectors of angular width $2^{-k}$ by defining
\begin{align*}
\Sigma_{N,k}^\omega \colonequals \bigg\{(\tau,\xi) \in \Sigma_N \colon \bigg\vert\frac{\xi}{|\xi|} - \omega\bigg\vert \leq 2^{-k}\bigg\}
\end{align*}
for each $\omega \in \bbS^{d-1}$.  We refer to these sets as \emph{$(N,k)$-sectors}.  Two $(N,k)$-sectors $\Sigma_{N,k}^\omega$ and $\Sigma_{N,k}^{\omega'}$ are \emph{related} if (i) $2^{2-k} \leq |\omega - \omega'| \leq 2^{4-k}$ and $k < N$, or (ii) $|\omega - \omega'| \leq 2^{4-N}$ and $k = N$.  The conic-type bilinear estimate that we require is the following:

\begin{lemma}\label{bilinear estimate for wide sectors}
Let $\Sigma_{N,k}^\omega$ and $\Sigma_{N,k}^{\omega'}$ be related $(N,k)$-sectors with $k < N$, and let $q > \frac{2(d+3)}{d+1}$.  Then
\begin{align*}
\|\fE f_1 \fE f_2\|_{q/2} \lesssim 2^{(N-k)(d-1-\frac{2(d+1)}{q})}\|f_1\|_2\|f_2\|_2
\end{align*}
whenever $\supp f_1 \subseteq \Sigma_{N,k}^\omega$ and $\supp f_2 \subseteq \Sigma_{N,k}^{\omega'}$.
\end{lemma}
\begin{proof}
After dividing $\Sigma_{N,k}^\omega$ and $\Sigma_{N,k}^{\omega'}$ into a bounded number of subsets with sufficiently small radial and angular width, one can directly apply \cite[Theorem 1.10]{Candy}.
\end{proof}

Next, we perform our bilinear-to-linear deduction.  To make it work, we will need additional bilinear estimates corresponding to related $(N,N)$-sectors, since these thinnest sectors are absent from Lemma \ref{bilinear estimate for wide sectors}.  Via a Lorentz boost, the local extension theory will be sufficient: 
   
\begin{lemma}\label{linear estimates for thin sectors}
If $\fE_0(p \rightarrow q)$ holds, then $\|\fE f\|_q \lesssim \|f\|_p$ whenever $f$ is supported in an $(N,N)$-sector.
\end{lemma}
\begin{proof}
Given an $(N,N)$-sector, a suitable Lorentz boost maps it into $\Sigma_0$.
\end{proof}

\begin{lemma}\label{bilinear-to-linear}
Suppose that $p$, $q$, $r$, and $\alpha$ relate in the following ways: \textnormal{(i)} $r \leq p \leq q \leq 4$, \textnormal{(ii)} $\alpha < 0$, \textnormal{(iii)} $\alpha \leq (d-1)(\frac{2}{p}-\frac{2}{r})$, and  \textnormal{(iv)} $\alpha \neq (d-1)(\frac{2}{q}-\frac{2}{r})$ or $p < q$.  Additionally, suppose that $\fE_0(p \rightarrow q)$ holds and that
\begin{align}\label{conditional bilinear estimate}
\|\fE f_1\fE f_2\|_{q/2} \lesssim 2^{(N-k)\alpha}\|f_1\|_r\|f_2\|_r
\end{align}
whenever $f_1$ and $f_2$ are supported in related $(N,k)$-sectors with $k < N$.  Then
\begin{align*}
\|\fE f\|_q \lesssim \|f\|_p
\end{align*}
whenever $|f| \sim \chi_\Omega$ for some $\Omega$ contained in some $\Sigma_N$ with $N \geq 0$.
\end{lemma}
\begin{proof}
As noted above, we will adapt the bilinear-to-linear argument found in \cite{Tao--Vargas--Vega}.  Since $\fE_0(p \rightarrow q)$ holds, we may fix $N \geq 1$.  Our first step is to construct a Whitney decomposition of $\Sigma_N \times \Sigma_N$.  For each $k \in \{1,\ldots,N\}$, choose a (finite) set $\Lambda_k \subset \bbS^{d-1}$ satisfying
\begin{align*}
\bbS^{d-1} = \bigcup_{\omega \in \Lambda_k} \{\omega' \in \bbS^{d-1} \colon |\omega - \omega'| \leq 2^{-k}\}
\end{align*}
and $|\omega - \omega'| \gtrsim 2^{-k}$ for all distinct $\omega,\omega' \in \Lambda_k$.  Given $\omega,\omega' \in \Lambda_k$, we write $\omega \sim \omega'$ if (i) $2^{2-k} \leq |\omega - \omega'| \leq 2^{4-k}$ and $k < N$, or (ii) $|\omega - \omega'| \leq 2^{4-N}$ and $k = N$.  (That is, $\omega \sim \omega'$ exactly when $\Sigma_{N,k}^\omega$ and $\Sigma_{N,k}^{\omega'}$ are related.)  We claim that
\begin{align}\label{Whitney decomposition}
\Sigma_N \times \Sigma_N = \bigcup_{k=1}^N\bigcup_{\substack{\omega,\omega' \in \Lambda_k \colon\\ \omega \sim \omega'}}\Sigma_{N,k}^\omega \times \Sigma_{N,k}^{\omega'}.
\end{align}
Indeed, fix $(\tau,\xi),(\tau',\xi') \in \Sigma_N$, and let $\zeta \colonequals \frac{\xi}{|\xi|}$ and $\zeta' \colonequals \frac{\xi'}{|\xi'|}$.  First, suppose that $|\zeta - \zeta'| \geq 12\cdot 2^{-N}$.  Then there exists $k \in \{1,\ldots,N-1\}$ and $\omega,\omega' \in \Lambda_k$ such that $6\cdot 2^{-k} \leq |\zeta - \zeta'| \leq 12\cdot 2^{-k}$ and $|\omega - \zeta|, |\omega' - \zeta'| \leq 2^{-k}$.  It follows that $(\tau,\xi) \in \Sigma_{N,k}^\omega$, $(\tau',\xi') \in \Sigma_{N,k}^{\omega'}$, and $\omega \sim \omega'$.  The case when $|\zeta - \zeta'| \leq 12\cdot 2^{-N}$ can be treated similarly using $k = N$, so the claim is proved.  The pieces of the decomposition \eqref{Whitney decomposition} are (unfortunately) not disjoint.  An easy argument shows, however, that two such pieces, $\Sigma_{N,k_1}^{\omega_1} \times \Sigma_{N,k_1}^{\omega_1'}$ and $\Sigma_{N,k_2}^{\omega_2} \times \Sigma_{N,k_2}^{\omega_2'}$, overlap only if (i) $|k_1 - k_2| \lesssim 1$ or (ii) $k_1 = k_2$ and $|\omega_1 - \omega_2|, |\omega_1' - \omega_2'| \lesssim 2^{-k_1}$.  Let
\begin{align*}
\fI \colonequals \{(k,\omega,\omega') \colon 1 \leq k \leq N;~\omega,\omega' \in \Lambda_k;~\omega \sim \omega'\},
\end{align*}
and let $C$ be a large constant.  We can partition $\fI$ into $O(1)$ sets $\fI_1,\ldots,\fI_n$ with the following separation property:  If $(k_1,\omega_1,\omega_1'),(k_2,\omega_2,\omega_2') \in \fI_j$, then either (i) $(k_1,\omega_1,\omega_1') = (k_2,\omega_2,\omega_2')$, or (ii) $k_1 = k_2$ and $\max\{|\omega_1-\omega_2|, |\omega_1'-\omega_2'|\} \geq C2^{-k_1}$, or (iii) $|k_1-k_2| \geq C$.  Thus, if $C$ is sufficiently large, then
\begin{align}\label{Whitney decomposition disjoint form}
\Sigma_N \times \Sigma_N = \bigcup_{j=1}^n\dot{\bigcup_{(k,\omega,\omega') \in \fI_j}} \Sigma_{N,k}^\omega \times \Sigma_{N,k}^{\omega'},
\end{align}
where the dot indicates a disjoint union.

From \eqref{Whitney decomposition disjoint form}, it follows that
\begin{align}\label{Whitney decomposition bound}
\|\fE f\|_q^2 = \|(\fE f)^2\|_{q/2} \lesssim \max_{1 \leq j \leq n}\bigg\|\sum_{(k,\omega,\omega') \in \fI_j}\fE(f\chi_{\Sigma_{N,k}^\omega})\fE(f\chi_{\Sigma_{N,k}^{\omega'}})\bigg\|_{q/2}.
\end{align}
By elementary geometry, there exist rectangles $R_{k,\omega,\omega'}$ such that $\Sigma_{N,k}^\omega + \Sigma_{N,k}^{\omega'} \subseteq R_{k,\omega,\omega'}$ and, for each $k$, the collection $\{2R_{k,\omega,\omega'}\}_{\omega,\omega' \in \Lambda_k \colon \omega \sim \omega'}$ has bounded overlap.  This fact allows us to exploit almost orthogonality in the form of \cite[Lemma 6.1]{Tao--Vargas--Vega}.  We obtain the bound
\begin{align}\label{almost orthogonality bound}
\bigg\|\sum_{(k,\omega,\omega') \in \fI_j}\fE(f\chi_{\Sigma_{N,k}^\omega})\fE(f\chi_{\Sigma_{N,k}^{\omega'}})\bigg\|_{q/2} \lesssim \sum_{k=1}^N\bigg(\sum_{\substack{\omega,\omega' \in \Lambda_k \colon\\ \omega \sim \omega'}}\|\fE(f\chi_{\Sigma_{N,k}^\omega})\fE(f\chi_{\Sigma_{N,k}^{\omega'}})\|_{q/2}^{q/2}\bigg)^\frac{2}{q}
\end{align}
for each $j$.  To help us estimate \eqref{almost orthogonality bound}, we set
\begin{align*}
\tilde{\Sigma}_{N,k}^\omega \colonequals \bigcup_{\substack{\omega' \in \Lambda_k\colon\\ \omega' \sim \omega \text{ or } \omega' = \omega}}\Sigma_{N,k}^{\omega'}
\end{align*}
and note that, for each $k$, the collection $\{\tilde{\Sigma}_{N,k}^\omega\}_{\omega \in \Lambda_k}$ has bounded overlap.  We first consider the terms in \eqref{almost orthogonality bound} with $k < N$.  Using \eqref{conditional bilinear estimate} and assuming that $|f| \sim \chi_\Omega$ for some $\Omega \subseteq \Sigma_N$, we see that
\begin{align}\label{terms with k<N}
\notag\sum_{k=1}^{N-1}\bigg(\sum_{\substack{\omega,\omega' \in \Lambda_k \colon\\ \omega \sim \omega'}}\|\fE(f\chi_{\Sigma_{N,k}^\omega})\fE(f\chi_{\Sigma_{N,k}^{\omega'}})\|_{q/2}^{q/2}\bigg)^\frac{2}{q} \notag&\lesssim \sum_{k=1}^{N-1}2^{(N-k)\alpha}\bigg(\sum_{\substack{\omega,\omega' \in \Lambda_k\colon \\ \omega \sim \omega'}}\mu(\Omega \cap \Sigma_{N,k}^\omega)^\frac{q}{2r}\mu(\Omega \cap \Sigma_{N,k}^{\omega'})^\frac{q}{2r}\bigg)^\frac{2}{q}\\
\notag&\lesssim \sum_{k=1}^{N-1}2^{(N-k)\alpha}\max_{\omega \in \Lambda_k}\mu(\Omega \cap \tilde{\Sigma}_{N,k}^\omega)^{\frac{2}{r}-\frac{2}{q}}\bigg(\sum_{\omega \in \Lambda_k}\mu(\Omega \cap \tilde{\Sigma}_{N,k}^\omega)\bigg)^\frac{2}{q}\\
&\lesssim \sum_{k=1}^{N-1} 2^{(N-k)\alpha}\min\{\mu(\Omega), 2^{(N-k)(d-1)}\}^{\frac{2}{r}-\frac{2}{q}}\mu(\Omega)^\frac{2}{q}.
\end{align}
If $\mu(\Omega) \leq 2^{d-1}$, then the hypotheses that $\alpha < 0$ and $r \leq p$ imply that \eqref{terms with k<N} is $O(\mu(\Omega)^\frac{2}{p})$.  Thus, we may assume that $\mu(\Omega) \geq 2^{d-1}$, and \eqref{terms with k<N} becomes
\begin{align*}
\sum_{k=1}^{N - \big\lceil\log_2 \mu(\Omega)^\frac{1}{d-1}\big\rceil}2^{(N-k)\alpha}\mu(\Omega)^\frac{2}{r} + \sum_{k = N - \big\lceil\log_2 \mu(\Omega)^\frac{1}{d-1}\big\rceil + 1}^{N-1}2^{(N-k)(\alpha - (d-1)(\frac{2}{q}-\frac{2}{r}))}\mu(\Omega)^\frac{2}{q}.
\end{align*}
The first sum is $O(\mu(\Omega)^\frac{2}{p})$ by the hypotheses that $\alpha < 0$ and $\alpha \leq (d-1)(\frac{2}{p}-\frac{2}{r})$. Treating separately the cases where $\alpha$ is strictly less than, strictly greater than, or equal to $(d-1)(\frac{2}{q}-\frac{2}{r})$, the second sum is similarly seen to be $O(\mu(\Omega)^\frac{2}{p})$.  Thus, altogether the terms in \eqref{almost orthogonality bound} with $k < N$ contribute $O(\mu(\Omega)^\frac{2}{p})$.  Now we estimate the terms with $k = N$.  By the Cauchy--Schwarz inequality, the hypothesis $\fE_0(p \rightarrow q)$, Lemma \ref{linear estimates for thin sectors}, and the hypothesis that $p \leq q$, we find that
\begin{align*}
\bigg(\sum_{\substack{\omega,\omega' \in \Lambda_N\colon\\ \omega \sim \omega'}}\|\fE(h\chi_{\Sigma_{N,N}^\omega})\fE(h\chi_{\Sigma_{N,N}^{\omega'}})\|_{q/2}^{q/2}\bigg)^\frac{2}{q} &\lesssim \bigg(\sum_{\substack{\omega,\omega' \in \Lambda_N\colon\\ \omega \sim \omega'}} \mu(\Omega \cap \Sigma_{N,N}^\omega)^\frac{q}{2p}\mu(\Omega \cap \Sigma_{N,N}^{\omega'})^\frac{q}{2p}\bigg)^\frac{2}{q}\\
&\lesssim\bigg(\sum_{\omega \in \Lambda_N} \mu(\Omega \cap \tilde{\Sigma}_{N,N}^\omega)^\frac{q}{p}\bigg)^\frac{2}{q}\\
&\lesssim \mu(\Omega)^\frac{2}{p}.
\end{align*}
Thus, we have shown that \eqref{almost orthogonality bound} is $O(\mu(\Omega)^\frac{2}{p})$.  Inserting this bound into \eqref{Whitney decomposition bound} and noting that $\|f\|_p \sim \mu(\Omega)^\frac{1}{p}$, the proof is complete.
\end{proof}

\begin{corollary}\label{unconditional result on frusta}
If $q > \frac{2(d+3)}{d+1}$ and $(\frac{d}{d+2}q)' \leq p \leq \min\{(\frac{d-1}{d+1}q)', q\}$ and $(p,q) \neq (\frac{2d}{d-1},\frac{2d}{d-1})$, then $\|\fE f\|_q \lesssim \|f\|_p$ whenever $f$ is supported in $\Sigma_N$ for some $N \geq 0$.
\end{corollary}
\begin{proof}
We will apply Lemma \ref{bilinear-to-linear}. By interpolation, we may assume that $q \leq \frac{2(d+2)}{d}$.  Then conditions (i)--(iv) in the lemma are satisfied with $r = 2$ and $\alpha = d-1-\frac{2(d+1)}{q}$.  The estimate $\fE_0(p \rightarrow q)$ is a consequence of the techniques in \cite{Tao}, and the bilinear estimate \eqref{conditional bilinear estimate} is valid by Lemma \ref{bilinear estimate for wide sectors}.  Thus, Lemma \ref{bilinear-to-linear} gives the restricted strong type analogue of the required estimate, and real interpolation completes the proof.
\end{proof}

\section{Summing bounds on frusta and proof of Theorem \ref{unconditional result}}\label{summing uniform bounds}
Let $\fE_{\fru}(p \rightarrow q)$ denote the statement that $\|\fE f\|_q \lesssim \|f\|_p$ whenever $f$ is supported in $\Sigma_N$ for some $N \geq 0$.  We have shown, by Corollary \ref{unconditional result on frusta}, that $\fE_{\fru}(p \rightarrow q)$ holds for $(p,q)$ in (a superset of) the range required by Theorem \ref{unconditional result}.  In this section, we sum these uniform bounds and consequently prove Theorem \ref{unconditional result}.  Our argument will utilize the following Strichartz estimate for the Klein--Gordon equation (see \cite{Kato--Ozawa} and references therein):  If $r \in [2,\infty]$, $s \in [2,\frac{2d}{d-2}]$ (with $\frac{2d}{d-2} \colonequals \infty$ when $d = 2$), $(r,s) \neq (2,\infty)$, and $\frac{1}{r} = \frac{d-1+\theta}{2}(\frac{1}{2}-\frac{1}{s})$ for some $\theta \in [0,1]$, then
\begin{align}\label{Strichartz estimate}
\|\fE f\|_{L_t^rL_x^s} \lesssim \|\langle \cdot \rangle^{\frac{1}{r}-\frac{1}{s}}f\|_2.
\end{align}

\begin{lemma}\label{summing over annuli}
If $\fE_{\fru}(p_0 \rightarrow q_0)$ holds for some $(\frac{d}{d+2}q_0)' \leq p_0 \leq \min\{(\frac{d-1}{d+1}q_0)', q_0\}$, then $\fE(p \rightarrow q)$ holds whenever $q > q_0$ and $p' = \frac{p_0'}{q_0}q$.
\end{lemma}
\begin{proof}
We will show that the hypothesis of the lemma implies the following bilinear estimate:  Given $q > q_0$ and $p' = \frac{p_0'}{q_0}q$, there exists a constant $c > 0$ such that
\begin{align}\label{bilinear estimate}
\|\fE f_1 \fE f_2\|_{q/2} \lesssim 2^{-c|N_1-N_2|}\|f_1\|_p\|f_2\|_p
\end{align}
whenever $\supp f_1 \subseteq \Sigma_{N_1}$ and $\supp f_2 \subseteq \Sigma_{N_2}$  for some $N_1,N_2 \geq 0$.  Assuming the validity of \eqref{bilinear estimate}, we now demonstrate how $\fE(p \rightarrow q)$ follows.  The case $q = \infty$ is trivial, so we assume that $q < \infty$ and set $n \colonequals \lceil q/2 \rceil$. Fixing $f$, we have
\begin{align*}
\|\fE f\|_q^q = \bigg\|\sum_{N=0}^\infty \fE(f\chi_{\Sigma_N})\bigg\|_q^q \leq \sum_{N_1,\ldots,N_{2n} \geq 0}\bigg\|\prod_{j=1}^{2n}\fE(f\chi_{\Sigma_N})\bigg\|_\frac{q}{2n}^\frac{q}{2n}
\end{align*}
since $q \leq 2n$.  Given $\bN \in \{0,1,2\ldots\}^{2n}$, let $p(\bN) = (p_j(\bN))_{j=1}^{2n}$ be a permutation of $\bN$ such that $\|f\|_{L^p(\Sigma_{p_1(\bN)})}$ is maximal among $\|f\|_{L^p(\Sigma_{p_j(\bN)})}$ and $|p_1(\bN)-p_2(\bN)|$ is maximal among $|p_1(\bN)-p_j(\bN)|$.  Then, by H\"older's inequality, estimate \eqref{bilinear estimate}, and the fact that $p \leq q$ (which follows from our hypothesis), we see that
\begin{align*}
\sum_{N_1,\ldots,N_{2n} \geq 0}\bigg\|\prod_{j=1}^{2n}\fE(f\chi_{\Sigma_N})\bigg\|_\frac{q}{2n}^\frac{q}{2n} &\lesssim \sum_{\bN \colon p(\bN) = \bN} \bigg\|\prod_{j=1}^{2n}\fE(f\chi_{\Sigma_N})\bigg\|_\frac{q}{2n}^\frac{q}{2n}\\
&\leq \sum_{\bN \colon p(\bN) = \bN}\prod_{j=1}^n\|\fE(f\chi_{\Sigma_{p_{2j-1}(\bN)}})\fE(f\chi_{\Sigma_{p_{2j}(\bN)}})\|_\frac{q}{2}^\frac{q}{2n}\\
&\lesssim \sum_{\bN \colon p(\bN) = \bN} 2^{-\frac{cq}{2n}|p_1(\bN) - p_2(\bN)|}\|f\|_{L^p(\Sigma_{p_1(\bN)})}^q\\
&\lesssim \sum_{N_1,N_2 \geq 0}|N_1-N_2|^{2n-2}2^{-\frac{cq}{2n}|N_1-N_2|}\|f\|_{L^p(\Sigma_{N_1})}^q\\
&\lesssim \sum_{N_1 \geq 0}\|f\|_{L^p(\Sigma_{N_1})}^q\\
&\leq \|f\|_p^q.
\end{align*}
Thus, we have shown that $\fE(p \rightarrow q)$ holds.

We turn to the proof of \eqref{bilinear estimate}.  If $N_1 = N_2$, then the desired estimate is a consequence of $\fE_{\fru}(p_0 \rightarrow q_0)$, the Cauchy--Schwarz inequality, and interpolation.  Thus, we may assume that $N_1 < N_2$; in particular, we have $N_2 \geq 1$.  Now, let $q_1 \colonequals \frac{2q_0}{p_0'}$ and choose $r_1,s_1,r_2,s_2$ obeying the conditions $r_i \in [2,\infty]$, $s_i \in [2,\frac{2d}{d-2}]$ (with $\frac{2d}{d-2} \colonequals \infty$ when $d=2$), $(r_i,s_i) \neq (2,\infty)$, and $\frac{2}{r_i} + \frac{2p_0'}{(q_0-p_0')s_i} = \frac{p_0'}{q_0-p_0'}$, as well as $r_1 < s_1$ and $\frac{2}{q_1} = \frac{1}{r_1}+\frac{1}{r_2} = \frac{1}{s_1}+\frac{1}{s_2}$.  (For example, fixing an arbitrary $r_1 \in [\frac{d(q_0-p_0')}{p_0'}, q_1)$ determines such a choice.) The Strichartz estimate \eqref{Strichartz estimate} and our hypothesis imply that $\|\fE f\|_{L_t^{r_i}L_x^{s_i}} \lesssim \|\langle \cdot \rangle^{\frac{1}{r_i}-\frac{1}{s_i}}f\|_2$ for every $f$.  Thus, by the mixed-norm Cauchy--Schwarz inequality, we have
\begin{align}\label{L^2 estimate}
\notag\|\fE f_1\fE f_2\|_{q_1/2} &\lesssim \|\fE f_1\|_{L_t^{r_1}L_x^{s_1}}\|\fE f_2\|_{L_t^{r_2}L_x^{s_2}}\\ &\lesssim 2^{N_1(\frac{1}{r_1}-\frac{1}{s_1})}2^{N_2(\frac{1}{r_2}-\frac{1}{s_2})}\|f_1\|_2\|f_2\|_2 = 2^{-(\frac{1}{r_1}-\frac{1}{s_1})|N_1-N_2|}\|f_1\|_2\|f_2\|_2.
\end{align}
The estimate \eqref{bilinear estimate} now follows by interpolating \eqref{L^2 estimate} with either the trivial inequality $\|\fE f_1 \fE f_2\|_\infty \lesssim \|f_1\|_1\|f_2\|_1$, if $q \geq q_1$, or the estimate $\|\fE f_1 \fE f_2\|_{q_0/2} \lesssim \|f_1\|_{p_0}\|f_2\|_{p_0}$ (a consequence of our hypothesis), if $q < q_1$.  
\end{proof}

{\bf Proof of Theorem \ref{unconditional result}.} Together, Corollary \ref{unconditional result on frusta} and Lemma \ref{summing over annuli} imply the theorem, except for the estimates $\fE(q \rightarrow q)$ with $\frac{10}{3} < q <4$ when $d = 2$.  The latter bounds can be obtained by (straightforwardly) adapting the proof of \cite[Lemma 8.2]{Bruce--Oliveira e Silva--Stovall}. \qed

\section{Conic decoupling and proof of Theorem \ref{conditional result}}\label{decoupling}

In this section, we prove our conditional result, Theorem \ref{conditional result}.  We will argue as follows:  To prove global extension estimates, it suffices to obtain uniform estimates on dyadic frusta, according to Lemma \ref{summing over annuli}.  By Lemma \ref{bilinear-to-linear}, these bounds would follow from appropriate bilinear estimates between $(N,k)$-sectors.  Lemma \ref{bilinear estimate for wide sectors} provides one such bilinear estimate, with a very favorable constant (relative to the hypotheses of Lemma \ref{bilinear-to-linear}) but valid only for $q$ in the bilinear range.  As we will show, the conic decoupling theorem of \cite{Bourgain--Demeter} and the hypothesis of Theorem \ref{conditional result} together imply a second bilinear estimate, with a worse constant but a smaller value of $q$.  Interpolation then leads to a compromise, wherein Lemma \ref{bilinear-to-linear} may be applied for a small set of exponents that nevertheless improves on the bilinear range.  After some arithmetic, the admissible exponents work out to be those satisfying \eqref{condition on exponents}.

We now turn to the details, beginning with the following elementary estimate:

\begin{lemma}\label{angular separation estimate}
If $u \in \R^n$ for some $n \geq 1$ and $x,y \in \R^2 \setminus \{0\}$ with $|x|,|y| \geq |u|$, then
\begin{align*}
\bigg\vert\frac{(x,u)}{|(x,u)|} - \frac{(y,u)}{|(y,u)|}\bigg\vert \geq \frac{1}{4}\bigg\vert\frac{x}{|x|}-\frac{y}{|y|}\bigg\vert.
\end{align*}
\end{lemma}

\begin{proof}
By a rotation of $\R^2$, we may assume that $y_1 > 0$ and $y_2 = 0$.  Let $\theta \in [-\pi,\pi]$ denote the angle between the vectors $x$ and $(1,0)$, so that $x = |x|(\cos\theta,\sin\theta)$.  Noting that $|\theta| \geq \vert\frac{x}{|x|}-\frac{y}{|y|}\vert$, it suffices to show that
\begin{align}\label{lower bound by theta}
\bigg\vert\frac{(x,u)}{|(x,u)|} - \frac{(y,u)}{|(y,u)|}\bigg\vert \geq \frac{|\theta|}{4}.
\end{align}
We find, by a bit of algebra, that
\begin{align*}
\bigg\vert\frac{(x,u)}{|(x,u)|} - \frac{(y,u)}{|(y,u)|}\bigg\vert^2 = \frac{2(|(x,u)||(y,u)| - |x||y|\cos(\theta) - |u|^2)}{|(x,u)||(y,u)|}.
\end{align*}
Due to the bound $\cos \theta \leq 1-\frac{\theta^2}{2}+\frac{\theta^4}{24}$ (which follows from Taylor's theorem), the Cauchy--Schwarz inequality, and the hypothesis that $|x|,|y| \geq |u|$, the right-hand side is bounded below by
\begin{align*}
\frac{2(|(x,u)||(y,u)| - |x||y| - |u|^2)}{|(x,u)||(y,u)|} + \frac{2|x||y|}{|(x,u)||(y,u)|}\bigg(\frac{\theta^2}{2}-\frac{\theta^4}{24}\bigg) \geq \frac{\theta^2}{2}-\frac{\theta^4}{24} \geq \frac{\theta^2}{16},
\end{align*}
completing the proof.
\end{proof}

The following consequence of conic decoupling is the technical heart of this section:

\begin{lemma}\label{decoupling lemma}
Suppose that $\fE_0(p \rightarrow q)$ holds for some $p \geq 2$ and $2 \leq q \leq \frac{2(d+1)}{d-1}$.  Then
\begin{align*}
\|\fE f\|_q \lesssim_\varepsilon 2^{(N-k)((d-1)(\frac{1}{2}-\frac{1}{p})+\varepsilon)}\|f\|_p
\end{align*}
for every $\varepsilon > 0$ whenever $f$ is supported in an $(N,k)$-sector.
\end{lemma}
\begin{proof}
By rotational symmetry, it suffices to prove the lemma for functions $f$ supported in the sector $\Sigma_{N,k}^{e_1}$.  We may also assume that $k > C$ and $N-k > C$, where $C$ is a positive integer of our choice.  Indeed, if $k \leq C$, then we can cover $\Sigma_{N,k}^{e_1}$ by a bounded number of $(N,C)$-sectors.  Similarly, if $N-k \leq C$, then $\Sigma_{N,k}^{e_1}$ is covered by a bounded number of $(N,N)$-sectors, and the required estimate is a consequence of the hypothesis $\fE_0(p \rightarrow q)$ and Lemma \ref{linear estimates for thin sectors}.

We proceed by rescaling the extension estimate on $\Sigma_{N,k}^{e_1}$ to one on a nearly conic set of unit size. There, the conic decoupling theorem, \cite[Theorem 1.2]{Bourgain--Demeter}, can be directly applied.  Let ${e}_1,\ldots,{e}_d$ denote the standard basis vectors in $\R^d$, and let $M \colonequals LD$, where $D$ is the conic dilation $D(\tau,\xi) \colonequals 2^{-N}(\tau,\xi)$ and $L$ is the linear map satisfying
\begin{align*}
L(0,e_j) &= 2^{(k-C)}(0,e_j),\quad\quad 2 \leq j \leq d,\\
L(1,e_1) &= (1,e_1),\\
L(-1,e_1) &= 2^{2(k-C)}(-1,e_1).
\end{align*}
One easily checks that $D(\Sigma_{N,k}^{e_1})$ lies in an $O(2^{-2N})$-neighborhood of the conic sector
\begin{align*}
\Gamma \colonequals \bigg\{(|\xi|,\xi) \colon 1 \leq |\xi| \leq 2,~\bigg\vert\frac{\xi}{|\xi|} - e_1\bigg\vert \leq 2^{-k}\bigg\}.
\end{align*}
The vectors $(0,e_2),\ldots,(0,e_d)$ are angularly tangent to the cone at $(1,e_1)$, while the vector $(1,e_1)$ is radially tangent and $(-1,e_1)$ is normal.  The map $L$ preserves the cone and expands $\Gamma$ to a sector of (roughly) unit angular width contained in the frustum
\begin{align*}
\tilde{\Gamma} \colonequals \{(|\xi|,\xi) \colon 1 \leq |\xi| \leq 3\}.
\end{align*}
Setting $\delta \colonequals 2^{2(k-N)}$ and assuming $C$ is sufficiently large, $M(\Sigma_{N,k}^{e_1})$ lies in the $\delta$-neighborhood of $\tilde{\Gamma}$.  Let $M_\ast \mu$ be the pushforward of $\mu$ by $M$, that is,
\begin{align*}
\int_{M(\Sigma)} g dM_\ast \mu \colonequals \int_\Sigma g \circ M d\mu,
\end{align*}
and let $\fE^M g \colonequals \widecheck{gdM_\ast \mu}$.  Let $\fP$ be a partition of the $\delta$-neighborhood of $\tilde{\Gamma}$ into plates of angular width $\delta^{1/2}$, thickness $\delta$, and length $1$, as in \cite[Theorem 1.2]{Bourgain--Demeter}.  Then
\begin{align}\label{decoupling consequence}
\|\fE^M g\|_q \lesssim_\varepsilon \delta^{-\varepsilon}\bigg(\sum_{\theta \in \fP}\|\fE^M(g\chi_{\theta'})\|_q^2\bigg)^\frac{1}{2}
\end{align}
for all $g$ supported in $M(\Sigma_{N,k}^{e_1})$, where $\theta' \colonequals \theta \cap M(\Sigma_{N,k}^{e_1})$.  Let $\fQ$ be a covering of $\Sigma_{N,k}^{e_1}$ by $(N,N)$-sectors having bounded overlap, and let $\{\psi_\kappa\}_{\kappa \in \fQ}$ be partition of unity with $\supp \psi_\kappa \subseteq \kappa$.  We claim that each $\theta \in \fP$ obeys the bound
\begin{align}\label{bounded number of intersections}
\#\{\kappa \in \fQ \colon \kappa \cap M^{-1}(\theta') \neq \emptyset\} \lesssim 1.
\end{align}
If \eqref{bounded number of intersections} holds, then taking $g = f \circ M^{-1}$ in \eqref{decoupling consequence}, rescaling, applying the hypothesis $\fE_0(p \rightarrow q)$ using Lemma \ref{linear estimates for thin sectors}, and finally applying H\"older's inequality and summing, we find that
\begin{align*}
\|\fE f\|_q &\lesssim_\varepsilon \delta^{-\varepsilon}\bigg(\sum_{\theta \in \fP}\|\fE(f\chi_{M^{-1}(\theta')})\|_q^2\bigg)^\frac{1}{2}\\
&\lesssim \delta^{-\varepsilon}\bigg(\sum_{\theta \in \fP}\sum_{\kappa \in \fQ}\|\fE(f\psi_\kappa\chi_{M^{-1}(\theta')})\|_q^2\bigg)^\frac{1}{2}\\
&\lesssim \delta^{-\varepsilon}\bigg(\sum_{\theta \in \fP}\sum_{\kappa \in \fQ}\|f\psi_\kappa\|_{L^p(M^{-1}(\theta'))}^2\bigg)^\frac{1}{2}\\
&\lesssim 2^{(N-k)((d-1)(\frac{1}{2}-\frac{1}{p})+2\varepsilon)}\|f\|_p.
\end{align*}
Since $\varepsilon$ is arbitrary, the proof is complete modulo the claim \eqref{bounded number of intersections}.

Toward proving \eqref{bounded number of intersections}, fix $\theta \in \fP$ and let
\begin{align*}
n \colonequals \#\{\kappa \in \fQ \colon \kappa \cap M^{-1}(\theta') \neq \emptyset\}.
\end{align*} 
To avoid notational annoyances, we will assume that $d \geq 3$ for the remainder of this argument.  The case $d=2$ is similar, but easier, and essentially appears in \cite{Bruce--Oliveira e Silva--Stovall}.  Given two points $(\tau,\xi),(\tau',\xi') \in \R \times \R^d$, let
\begin{align*}
\dist_{\ang}((\tau,\xi),(\tau',\xi')) \colonequals \bigg\vert\frac{\xi}{|\xi|} - \frac{\xi'}{|\xi'|}\bigg\vert
\end{align*}
denote their angular separation.  Suppose that $\Sigma_{N,k}^{e_1} \cap M^{-1}(\theta)$ has angular width at most $n^{\frac{1}{2(d-1)}}2^{-N}$.  Since $\fQ$ has bounded overlap, it follows that
\begin{align*}
n \lesssim \frac{(n^\frac{1}{2(d-1)}2^{-N})^{d-1}}{2^{-N(d-1)}} = n^\frac{1}{2},
\end{align*}
and thus $n \lesssim 1$.  We may assume, therefore, that there exist points $(\tau,\xi),(\tau',\xi') \in \Sigma_{N,k}^{e_1} \cap M^{-1}(\theta)$ such that $\dist_{\ang}((\tau,\xi),(\tau',\xi')) \geq n^\frac{1}{2(d-1)}2^{-N}$.  Since $\theta$ has angular width $O(2^{k-N})$, it suffices to show that
\begin{align}\label{target estimate}
\dist_{\ang}(M(\tau,\xi),M(\tau',\xi')) \gtrsim 2^k\dist_{\ang}((\tau,\xi),(\tau',\xi')).
\end{align}
We proceed by exploiting symmetry.  First, we observe that
\begin{align}\label{radial scaling}
\dist_{\ang}(M(\tau,\xi),M(\tau',\xi')) = \dist_{\ang}(M(\langle \lambda^{-1}\rangle, \lambda^{-1}\xi), M(\langle (\lambda')^{-1}\rangle, (\lambda')^{-1}\xi')),
\end{align}
where $\lambda \colonequals |\xi|$, $\lambda' \colonequals |\xi'|$, and $\langle x \rangle \colonequals \sqrt{x^2+1}$.  Next, we utilize rotational invariance. Let $R \colonequals I_2 \oplus S$, where $I_2$ is the $2 \times 2$ identity matrix and $S$ is a rotation of $\R^{d-1}$ satisfying
\begin{align*}
S\bigg(\frac{\lambda^{-1}(\xi_2,\ldots,\xi_d)-(\lambda')^{-1}(\xi_2',\ldots,\xi_d')}{|\lambda^{-1}(\xi_2,\ldots,\xi_d)-(\lambda')^{-1}(\xi_2',\ldots,\xi_d')|}\bigg) = (1,0,\ldots,0).
\end{align*}
(One can check that $\lambda^{-1}(\xi_2,\ldots,\xi_d) \neq (\lambda')^{-1}(\xi_2',\ldots,\xi_d')$.)  The maps $M$ and $R$ commute, and $R$ (and thus $R^{-1}$) preserves angular separation.  Setting $(\rho,\zeta) \colonequals R(\langle \lambda^{-1} \rangle, \lambda^{-1}\xi)$ and $(\rho',\zeta') \colonequals R(\langle(\lambda')^{-1}\rangle, (\lambda')^{-1}\xi')$ and using \eqref{radial scaling}, we see that
\begin{align}\label{combined symmetry}
\dist_{\ang}(M(\tau,\xi),M(\tau',\xi')) = \dist_{\ang}(M(\rho,\zeta),M(\rho',\zeta')).
\end{align}
The definitions of $\Sigma_{N,k}^{e_1}$, $\lambda$, $\lambda'$, and $R$ imply the following:
\begin{itemize}
\item[(i)]{$|\zeta| = |\zeta'| = 1$ and $|\zeta - e_1|,|\zeta' - e_1| \leq 2^{-k}$;}
\item[(ii)]{$\dist_{\ang}((\tau,\xi),(\tau,\xi')) = |\zeta -\zeta'|$;}
\item[(iii)]{$\zeta_j = \zeta_j' =: a_j$ for all $j \in \{3,\ldots,d\}$;}
\item[(iv)]{$1 \leq \rho,\rho' \leq 1+ 2^{-2N}$;}
\end{itemize}
We write $(\rho,\zeta) = (\rho, r\cos\nu, r\sin\nu, a_3,\ldots, a_d)$, where $r \colonequals \sqrt{1-a_3^2 - \cdots - a_d^2} = \sqrt{\zeta_1^2+\zeta_2^2}$ and $\nu \colonequals \arctan(\zeta_2/\zeta_1)$, and we record that $1- 2^{-2k} \leq r \leq 1$ and $|\nu| \leq 2^{1-k}$ if $C$ is sufficiently large. We compute that $M(\rho,\zeta) = 2^{-N-1}(m_1(\rho,\nu), m_2(\rho,\nu),m_3(\rho,\nu),u)$, where
\begin{align*}
m_1(x,y) &\colonequals (1+2^{2(k-C)})x + (1-2^{2(k-C)})r\cos y,\\
m_2(x,y) &\colonequals (1-2^{2(k-C)})x + (1+2^{2(k-C)})r\cos y,\\
m_3(x,y) &\colonequals 2^{k-C+1}r\sin y,\\
u &\colonequals 2^{k-C+1}(a_3,\ldots,a_d).
\end{align*}
One easily checks that if $C$ is sufficiently large, then $|u| \leq 1$ and $1 \leq m_2(x,y) \leq 2$ and $|m_3(x,y)| \leq 2$ whenever $1 \leq x \leq 1+2^{-2N}$ and $|y| \leq 2^{1-k}$.  Writing, analogously, $(\rho',\zeta') = (\rho', r\cos\nu', r\sin\nu', a_3,\ldots,a_d)$ and using \eqref{combined symmetry} and Lemma \ref{angular separation estimate}, we find that
\begin{align}\label{reduction to angle difference}
\notag\dist_{\ang}(M(\tau,\xi),M(\tau',\xi')) &= \bigg\vert\frac{(m_2(\rho,\nu),m_3(\rho,\nu),u)}{|(m_2(\rho,\nu),m_3(\rho,\nu),u)|} - \frac{(m_2(\rho',\nu'),m_3(\rho',\nu'),u)}{|(m_2(\rho',\nu'),m_3(\rho',\nu'),u)|}\bigg\vert\\
\notag&\gtrsim \bigg\vert\frac{(m_2(\rho,\nu),m_3(\rho,\nu))}{|(m_2(\rho,\nu),m_3(\rho,\nu))|} - \frac{(m_2(\rho',\nu'),m_3(\rho',\nu'))}{|(m_2(\rho',\nu'),m_3(\rho',\nu'))|}\bigg\vert\\
&\sim|A(\rho,\nu) - A(\rho',\nu')|,
\end{align}
where
\begin{align*}
A(x,y) \colonequals \arctan\bigg(\frac{m_3(x,y)}{m_2(x,y)}\bigg).
\end{align*}
Using the mean value theorem and bounds on the $m_j$, we find that
\begin{align}\label{x difference}
|A(\rho,\nu)-A(\rho,\nu')| \geq |\nu-\nu'|\inf_{|y| \leq 2^{1-k}}|\partial_2A(\rho,y)| \gtrsim 2^{-C}2^k|\nu - \nu'|
\end{align}
and
\begin{align}\label{y difference}
|A(\rho,\nu')-A(\rho',\nu')| \leq |\rho-\rho'|\sup_{1 \leq x \leq 1+2^{-2N}}|\partial_1A(x,\nu')| \lesssim 2^{-2C}2^{-2N}2^{2k} \leq 2^{-3C}2^k2^{-N},
\end{align}
where the implicit constants do not depend on $C$.  Since
\begin{align*}
2^{-N} \lesssim \dist_{\ang}((\tau,\xi),(\tau',\xi')) = |\zeta - \zeta'| = |(\zeta_1,\zeta_2)-(\zeta_1',\zeta_2')| \sim |\nu - \nu'|,
\end{align*}
\eqref{target estimate} follows from \eqref{reduction to angle difference}--\eqref{y difference} after fixing $C$ sufficiently large.
\end{proof}

{\bf Proof of Theorem \ref{conditional result}.}  By Lemma \ref{summing over annuli}, it suffices to prove that $\fE_{\fru}(p \rightarrow q)$ holds for all $(p,q)$ satisfying the hypotheses of the theorem.  We have assumed that $\fE_0(p_0 \rightarrow q_0)$ holds with $p_0' = \frac{d}{d+2}q_0$ for some $q_0 < \frac{2(d+3)}{d+1}$.  Necessarily $p_0 \geq 2$, so by the Cauchy--Schwarz inequality and Lemma \ref{decoupling lemma}, we have
\begin{align*}
\|\fE f_1\fE f_2\|_{q_0/2} \lesssim_\varepsilon 2^{(N-k)((d-1)(1-\frac{2}{p_0})+2\varepsilon)}\|f_1\|_{p_0}\|f_2\|_{p_0}
\end{align*}
for every $\varepsilon > 0$ whenever $f_1$ and $f_2$ are supported in $(N,k)$-sectors.  Given $q_1 > \frac{2(d+3)}{d+1}$, we also have
\begin{align*}
\|\fE f_1 \fE f_2\|_{q_1/2} \lesssim 2^{(N-k)(d-1 - \frac{2(d+1)}{q_1})}\|f_1\|_2\|f_2\|_2
\end{align*}
by Lemma \ref{bilinear estimate for wide sectors}, provided $f_1$ and $f_2$ are supported in related $(N,k)$-sectors.  Interpolating these estimates, we see that
\begin{align*}
\|\fE f_1 \fE f_2\|_{q_t/2} \lesssim_\varepsilon 2^{(N-k)\alpha_t}\|f_1\|_{r_t}\|f_2\|_{r_t},
\end{align*}
where
\begin{align*}
\bigg(\frac{1}{r_t},\frac{1}{q_t}\bigg) &\colonequals (1-t)\bigg(\frac{1}{p_0},\frac{1}{q_0}\bigg) + t\bigg(\frac{1}{2},\frac{1}{q_1}\bigg),\\
\alpha_t &\colonequals (1-t)\bigg((d-1)\bigg(1-\frac{2}{p_0}\bigg)+2\varepsilon\bigg) + t\bigg(d-1-\frac{2(d+1)}{q_1}\bigg)
\end{align*}
for $t \in [0,1]$.  Given $q \in (q_0, \frac{2(d+3)}{d+1})$, let $t$ be such that $q = q_t$, and suppose that 
\begin{align}\label{initial condition on exponents}
\frac{1}{p} > \frac{\alpha_t}{2(d-1)} + \frac{1}{r_t}.
\end{align}
We may apply Lemma \ref{bilinear-to-linear} to obtain the estimate
\begin{align}\label{restricted strong type estimate}
\|\fE f\|_q \lesssim \|f\|_p
\end{align}
whenever $|f| \sim \chi_{\Omega}$ for some $\Omega$ contained in some $\Sigma_N$. Indeed, the hypotheses (i)--(iv) in Lemma \ref{bilinear-to-linear} hold with $r = r_t$ and $\alpha = \alpha_t$, and the estimate $\fE_0(p \rightarrow q)$ is a consequence of interpolating $\fE_0(p_0 \rightarrow q_0)$ and $\fE_0(1 \rightarrow \infty)$ and applying H\"older's inequality.   Letting $q_1 \rightarrow \frac{2(d+3)}{d+1}$ and $\varepsilon \rightarrow 0$, the condition \eqref{initial condition on exponents} becomes \eqref{condition on exponents}, and thus \eqref{restricted strong type estimate} extends to all $p$ satisfying \eqref{condition on exponents}.  Real interpolation now implies that $\fE_{\fru}(p \rightarrow q)$ holds in the required range.  \qed

\section{Possible improvements via local estimates for elliptic surfaces}\label{improvements}
In this final section, we discuss some likely improvements to Theorem \ref{unconditional result} by means of the state-of-the-art local extension estimates for elliptic surfaces (as defined in, e.g., \cite{Tao--Vargas--Vega}).  We will ignore some details for simplicity, and thus we do not claim any improvement definatively.  

As described in the introduction, the validity of the local estimate $\fE_0(p \rightarrow q)$ on the parabolic scaling line $p' = \frac{d}{d+2}q$ for some $q < \frac{2(d+3)}{d+1}$ would imply an improvement of Theorem \ref{unconditional result} by a direct application of our conditional result, Theorem \ref{conditional result}.  Such an estimate appears to follow from known results:  Let $\fE_{\ellip}^d(p \rightarrow q)$ denote the statement that for every elliptic phase $\phi : [-1,1]^d \rightarrow \R$, the associated extension operator
\begin{align*}
\fE_\phi f(t,x) \colonequals \int_{[-1,1]^d}e^{i(t,x)\cdot(\phi(\xi),\xi)}f(\xi)d\xi
\end{align*}
is bounded from $L^p([-1,1]^d)$ to $L^q(\R \times \R^d)$ with operator norm depending only on $p$, $q$, $d$, and the parameters used to define ellipticity.  (In particular, $\fE_{\ellip}^d(p \rightarrow q)$ would imply $\fE_0(p \rightarrow q)$.)  Hickman and Rogers \cite{Hickman--Rogers} have shown that for each $d \geq 2$, there exists some $q_d < \frac{2(d+3)}{d+1}$ such that  $\fE_{\ellip}^d(q\rightarrow q)$ holds whenever $q > q_d$.  (Their result is stated for paraboloids, but an adaptation of their methods yields estimates for general elliptic surfaces; see \cite[Remark 11.3]{Hickman--Rogers} and references therein.) One can move their estimate to the scaling line $p' = \frac{d}{d+2}q$ in a standard way, but with a loss in the range of $q$.  Namely, one interpolates the bilinear version of $\fE_{\ellip}^d(q \rightarrow q)$ (from the Cauchy--Schwarz inequality) with the $L^2$-based bilinear extension estimate for elliptic surfaces (see \cite{Tao}) and then obtains a linear estimate on the scaling line using the bilinear-to-linear method, \cite[Theorem 2.2]{Tao--Vargas--Vega}.  In the end, these steps reveal that $\fE_{\ellip}^d(p \rightarrow q)$ holds with $p' = \frac{d}{d+2}q$ whenever
\begin{align*}
q > \tilde{q}_d \colonequals \frac{(2q_d-4)d^2+(6q_d-16)d+2q_d-12}{(q_d-2)d^2+(q_d-4)d}.
\end{align*}
We note that $\tilde{q}_d < \frac{2(d+3)}{d+1}$ (a consequence of the interpolation), so an improvement of Theorem \ref{unconditional result} is indeed obtained.  Hickman and Rogers' exponent $q_d$ can be computed (see \cite[Footnote 5 and Figure 1]{Hickman--Rogers}), and when $d \neq 2,5,7,9,11$, it  defines the best known range of local extension estimates for $d$-dimensional elliptic surfaces.  Wang \cite{Wang} has the current record for $d=2$, while the best results for $d=5,7,9,11$ are due to Guth \cite{Guth2} (both articles study paraboloids).

Additionally, the method of slicing offers a means of improving Theorem \ref{unconditional result} on the conic scaling line $p' = \frac{d-1}{d+1}q$.  Since the cross sections of $\Sigma$ are $(d-1)$-dimensional spheres, it is possible to deduce certain extension estimates for $\Sigma$ using the boundedness of the extension operator associated to $\bbS^{d-1}$.  We have the following conditional result:

\begin{proposition}\label{slicing}
If $p' = \frac{d-1}{d+1}q$ and 
\begin{align}\label{sphere bilinear range}
\bigg\|\int_{\bbS^{d-1}}e^{ix\cdot\theta}f(\theta)d\sigma(\theta)\bigg\|_{L^q(\R^d)} \lesssim \|f\|_{L^p(\bbS^{d-1})}
\end{align}
for all $f \in L^p(\bbS^{d-1})$, then $\fE(p \rightarrow q)$ holds.
\end{proposition}
\begin{proof}
We proceed along the lines of arguments in \cite{Drury--Guo} and \cite{Nicola}.  For later use, we record that $q \geq \max\{2,p\}$ due to the hypothesis that \eqref{sphere bilinear range} holds.   Now, in polar coordinates, our extension operator takes the form
\begin{align*}
\fE f(t,x) = \int_0^\infty \int_{\bbS^{d-1}}e^{i(t,x)\cdot(\langle r \rangle, r\theta)}f(\langle r \rangle, r\theta)\frac{r^{d-1}}{\langle r \rangle}d\sigma(\theta)dr = \int_1^\infty e^{its}\int_{\bbS^{d-1}} e^{i\llangle s \rrangle x \cdot \theta}f(s, \llangle s \rrangle\theta)\llangle s \rrangle^{d-2}d\sigma(\theta)ds,
\end{align*}
where $\llangle s \rrangle \colonequals \sqrt{s^2-1}$.  Using the (dualized) Lorentz space version of the Hausdorff--Young inequality and a Minkowski-type inequality (see \cite[Corollary 3.16]{Stein--Weiss} and \cite[Lemma 2.1]{Nicola}, respectively), it follows that
\begin{align}\label{Hausdorff--Young, Minkowski}
\notag\|\fE f\|_q &\lesssim \bigg\|\bigg\|\int_{\bbS^{d-1}}e^{i\llangle s \rrangle x \cdot \theta}f(s,\llangle s \rrangle\theta)\llangle s \rrangle^{d-2}d\sigma(\theta)\bigg\|_{L_s^{q',q}}\bigg\|_{L_x^q}\\ &\lesssim \bigg\|\bigg\|\int_{\bbS^{d-1}}e^{i\llangle s \rrangle x \cdot \theta}f(s,\llangle s \rrangle\theta)\llangle s \rrangle^{d-2}d\sigma(\theta)\bigg\|_{L_x^q}\bigg\|_{L_s^{q',q}}.
\end{align}
By a change of variable and the estimate \eqref{sphere bilinear range}, the inner norm in \eqref{Hausdorff--Young, Minkowski} obeys the bound
\begin{align}\label{inner norm}
\bigg\|\int_{\bbS^{d-1}}e^{i\llangle s \rrangle x \cdot \theta}f(s,\llangle s \rrangle\theta)\llangle s \rrangle^{d-2}d\sigma(\theta)\bigg\|_{L_x^q} \lesssim \llangle s \rrangle^{d-2-\frac{d}{q}}\|f(s,\llangle s \rrangle\theta)\|_{L_\theta^p}.
\end{align}
Due to the embedding $L^{q',p} \hookrightarrow L^{q',q}$ and the Lorentz space version of H\"older's inequality (see \cite[Theorem 3.6]{O'Neil}), we have that
\begin{align}\label{Holder}
\notag\|\llangle s \rrangle^{d-2-\frac{d}{q}}\|f(s,\llangle s \rrangle\theta)\|_{L_\theta^p}\|_{L_s^{q',q}} &\lesssim \|\llangle s \rrangle^{d-2-\frac{d}{q}}\|f(s,\llangle s \rrangle\theta)\|_{L_\theta^p}\|_{L_s^{q',p}}\\
\notag&\lesssim \|\llangle s \rrangle^{-\frac{1}{\alpha}}\|_{L_s^{\alpha,\infty}}\|\llangle s \rrangle^{\frac{1}{\alpha}+d-2-\frac{d}{q}}\|f(s,\llangle s \rrangle\theta)\|_{L_\theta^p}\|_{L_s^{p,p}}\\
&\lesssim \|\llangle s \rrangle^{\frac{1}{\alpha}+d-2-\frac{d}{q}}\|f(s,\llangle s \rrangle\theta)\|_{L_\theta^p}\|_{L_s^{p}},
\end{align}
where $\frac{1}{\alpha} \colonequals \frac{1}{q'} - \frac{1}{p} = \frac{2}{(d-1)q}$.  By the change of variable $r \colonequals \llangle s \rrangle$ and some algebra, we find that
\begin{align}\label{change to original variable}
\|\llangle s \rrangle^{\frac{1}{\alpha}+d-2-\frac{d}{q}}\|f(s,\llangle s \rrangle\theta)\|_{L_\theta^p}\|_{L_s^{p}} = \bigg(\int_0^\infty\int_{\bbS^{d-1}}|f(\langle r \rangle, r\theta)|^p\frac{r^{d-1}}{\langle r \rangle}d\sigma(\theta)dr\bigg)^\frac{1}{p} = \|f\|_p.
\end{align}
Combining \eqref{Hausdorff--Young, Minkowski}--\eqref{change to original variable}, we conclude that $\fE(p \rightarrow q)$ holds.
\end{proof}

Since the sphere $\bbS^{d-1}$ is elliptic, \eqref{sphere bilinear range} holds in the range $q > \tilde{q}_{d-1}$ and $p' = \frac{d-1}{d+1}q$, as discussed above.  Proposition \ref{slicing} therefore yields an improvement to Theorem \ref{unconditional result} on the conic scaling line whenever
\begin{align*}
\tilde{q}_{d-1} < \frac{2(d+3)}{d+1}.
\end{align*}    
Modifying the code from \cite[Footnote 5]{Hickman--Rogers} to compute $q_{d-1}$ and then $\tilde{q}_{d-1}$, we determine that this inequality holds when $d=22,24,26$ or $28 \leq d \leq 100$.  We believe that it continues to hold for all $d > 100$; however, an asymptotic comparison of $\tilde{q}_{d-1}$ and $\frac{2(d+3)}{d+1}$ is not possible without further information about the asymptotics of $q_{d-1}$.  For the remaining low dimensions, $2 \leq d \leq 21$ and $d=23,25,27$, Hickman and Rogers' result does not improve Theorem \ref{unconditional result} on the conic line.  The stronger results of Wang \cite{Wang} and Guth \cite{Guth2} (conditionally extended to general elliptic surfaces) for $d = 2$ and $d = 5,7,9,11$ also do not yield such an improvement.


\end{document}